\documentclass[11pt,a4paper,twoside,reqno]{amsart}
\usepackage{amsfonts} %
\usepackage{amsmath} %
\usepackage{amssymb} %
\usepackage{amsthm} %
\usepackage{enumerate} %
\usepackage{bbm} %
\usepackage[T1]{fontenc} %
\usepackage[latin9]{inputenc} %
\usepackage{ae,aecompl} %
\usepackage{graphicx} %
\usepackage{nicefrac} %
\usepackage{color} %

\newtheorem{theorem}{Theorem}[section] %
\newtheorem{corollary}[theorem]{Corollary} %
\newtheorem{lemma}[theorem]{Lemma} %
\newtheorem*{method}{Computational method} %
\newtheorem{proposition}[theorem]{Proposition} %
{\theoremstyle{remark} %
  \newtheorem{remark}[theorem]{Remark}} %
{\theoremstyle{definition} %
  \newtheorem{definition}[theorem]{Definition} %
}

\newcommand{\PP}[0]{\ensuremath{\mathbb{P}}}
\newcommand{\CC}[0]{\ensuremath{\mathbb{C}}}
\newcommand{\ZZ}[0]{\ensuremath{\mathbb{Z}}}

\newcommand{\AAA}[0]{\ensuremath{\mathcal{A}}}

\newcommand{\fp}[0]{\ensuremath{\mathbb{F}}}
\newcommand{\FF}[0]{\ensuremath{\mathcal{F}}}
\newcommand{\TT}[0]{\ensuremath{\mathcal{T}}}

\newcommand{\EE}[0]{\ensuremath{\mathcal{E}}}
\newcommand{\JJ}[0]{\ensuremath{\mathcal{J}}}

\newcommand{\aut}[0]{\ensuremath{\operatorname{Aut}}}
\newcommand{\lin}[0]{\ensuremath{\operatorname{Lin}}}

\newcommand{\diag}[0]{\ensuremath{\operatorname{diag}}}
\newcommand{\PGL}[0]{\ensuremath{\operatorname{PGL}}}
\newcommand{\GL}[0]{\ensuremath{\operatorname{GL}}}
\newcommand{\PSL}[0]{\ensuremath{\operatorname{PSL}}}

\begin{document}

\title[]{Automorphisms of prime order of smooth cubic $n$-folds}

\author{V\'\i ctor Gonz\'alez-Aguilera and Alvaro Liendo}
\address{Departamento de Matem\'aticas, Universidad T\'ecnica
  Fe\-de\-ri\-co San\-ta Ma\-r\'\i a, Ca\-si\-lla 110-V, Valpara\'\i
  so, Chile.}
\email{vgonzale@mat.utfsm.cl}

\address{Mathematisches Institut, Universit\"at Basel, Rheinsprung 21,
  CH-4051 Basel, Switzerland.}
\email{alvaro.liendo@unibas.ch}

\date{\today}

\thanks{
  \mbox{\hspace{11pt}}{\it Key words}: Cubic $n$-folds, automorphisms
  of hypersurfaces, principally polarized abelian varieties.\\
  \mbox{\hspace{11pt}}The first author was partially supported by the
  Fondecyt project 1080030 and the Dgip of the UTFSM}

\begin{abstract}
  In this paper we give an effective criterion as to when a prime
  number $p$ is the order of an automorphism of a smooth cubic
  hypersurface of $\PP^{n+1}$, for a fixed $n\geq 2$. We also provide
  a computational method to classify all such hypersurfaces that admit
  an automorphism of prime order $p$. In particular, we show that
  $p<2^{n+1}$ and that any such hypersurface admitting an automorphism
  of order $p>2^n$ is isomorphic to the Klein $n$-fold. We apply our
  method to compute exhaustive lists of automorphism of prime order of
  smooth cubic threefolds and fourfolds. Finally, we provide an
  application to the moduli space of principally polarized abelian
  varieties.
\end{abstract}

\maketitle

\section*{Introduction}

The smooth cubic hypersurfaces of the projective space $\PP^{n+1}$ (or
cubic $n$-fold for short), $n\geq 2$ are classical objects in
algebraic geometry. Its groups of regular automorphisms are finite and
induced by linear automorphisms of $\PP^{n+1}$ \cite{mat}. In the case
$n=2$ they correspond to the classical cubic surfaces. Segre in $1942$
used the geometry of its $27$ lines produced a list of cubic surfaces
admitting non-trivial automorphism group \cite{segre}, see also
\cite{dolga}.

% For the cubic threefold and $5$-fold, their intermediate jacobians
% are $5$-dimensional and $21$-dimensional principally polarized
% abelian varieties (p.p.a.v.), respectively. Hence, their
% automorphisms groups are related with the singular locus of the
% moduli space of $5$-dimensional and $21$-dimensional p.p.a.v.

% In \cite{allcock} a remarkable construction is given. Using a family
% of cubic threefolds admitting an automorphism of order three, it is
% proved that the moduli space of stable cubic surfaces is isomorphic
% to a quotient of the complex hyperbolic $4$-space by the action of a
% discrete group. A similar idea is used in \cite{allcock2} for the
% moduli space of stable cubic threefolds.

Let $n\geq 2$ be a fixed integer. The main result of this paper is the
following criterion for the order of an automorphism of a cubic
$n$-fold: a prime number $p$ is the order of an automorphism of a
smooth cubic hypersurface of $\PP^{n+1}$ if an only if $p=2$ or there
exists $\ell\in\{1,\ldots,n+2\}$ such that $(-2)^\ell\equiv 1\mod
p$. See Theorem~\ref{resultado} in Section~\ref{main}.

In Corollary 1.8 we apply this criterion to show that if a prime
number $p$ is the order of an automorphism of a cubic $n$-fold, then
$p<2^{n+1}$. This bound is sharper that the general bound in
\cite{szabo} specialized to cubic hypersurfaces. Moreover, we show
that any $n$-fold admitting an automorphism of prime order $p>2^n$ is
isomorphic to the Klein $n$-fold given by the cubic form
$F=x_0^2x_1+x_1^2x_2+\ldots+x_n^2x_{n+1}+x_{n+1}^2x_0$.

In Section~\ref{class} we develop a computational method to classify
all the smooth cubic $n$-folds admitting an automorphism of order $p$
prime. We illustrate our method in the cases of threefolds and
fourfolds, see Theorems~\ref{tthre} and \ref{tfour}.

% Some applications of our results are provided in
% Section~\ref{abvar}.  First, applying a theorem in \cite{roulleau},
% we give the possible prime numbers that can be the order of an
% automorphism of the Fano surface of a cubic threefold. These exclude
% the prime number $7$ that is not, a priori, excluded in
% \cite{roulleau}. Second, we show a way to construct some families of
% $K3$ surfaces admitting a symplectic automorphism, this is based on
% a result due to Mukai~\cite{mukai}.

Finally, in Section~\ref{abvar} we apply a theorem in \cite{victor} to
show that the intermediate jacobian of the Klein $5$-fold is a
zero-dimensional non-isolated component of the singular locus of the
moduli space of principally polarized abelian varieties of dimension
21. The analog result for the Klein threefold was proved in
\cite{victor}.

\section{Admissible prime orders of automorphisms of a cubic $n$-fold}
\label{main}

In this section we give a criterion as to when a prime number $p$
appear as the order of an automorphism of a smooth cubic $n$-fold in
$\PP^{n+1}$, $n\geq 2$. We also give a simple bound for the order of
any such automorphism.

Letting $V$ be a vector space over $\CC$ of dimension $n+2$, $n\geq2$
with a fixed basis $\{x_0,\ldots,x_{n+1}\}$, let $\PP^{n+1}=\PP(V)$ be
the corresponding projective space, and $S^3(V)$ be the vector space
of cubic forms on $V$. We fix the basis $\{x_ix_jx_k: 0\leq i\leq j
\leq k \leq n+1\}$ of $S^3(V)$. The dimension of $S^3(V)$ is
$\binom{n+4}{3}$.

For a cubic form $F\in S^3(V)$, we denote by
$X=V(F)\subseteq\PP^{n+1}$ the cubic hypersurface of dimension $n$
(cubic $n$-fold for short) associated to $F$. We denote by $\aut(X)$
the group of regular automorphisms of $X$ and by $\lin(X)$ the
subgroup of $\aut(X)$ that extends to automorphisms of $\PP^{n+1}$. By
Theorems 1 and 2 in \cite{mat} if $X$ is smooth, then
$$\aut(X)=\lin(X) \quad\mbox{and}\quad |\aut(X)|<\infty.$$

In this setting $\aut(X)<\PGL(V)$ and for any automorphism in
$\aut(X)$ we can choose a representative in $\GL(V)$. This
automorphism induces an automorphism of $S^3(V)$ such that
$\varphi(F)=\lambda F$. These three automorphisms will be denoted by
the same letter $\varphi$.

In this paper we consider automorphisms of order $p$ prime. In this
case, multiplying by an appropriate constant, we can assume that
$\varphi^p=\operatorname{Id}_V$, so that $\varphi$ is also a linear
automorphism of finite order $p$ of $V$ and $\varphi(F)=\xi^{a} F$
where $\xi$ is a $p$-th root of the unity. Furthermore, we can apply a
linear change of coordinates on $V$ to diagonalize $\varphi$, so that
$$\varphi:V\rightarrow V,\qquad (x_0,\ldots,x_{n+1})\mapsto(\xi^{\sigma_0}x_0,\ldots,\xi^{\sigma_{n+1}}x_{n+1}), \qquad 0\leq\sigma_i<p.$$
\begin{definition}
  We define the signature $\sigma$ of an automorphism $\varphi$ as
  above by
  $$\sigma=(\sigma_{0},\ldots,\sigma_{n+1})\in \fp_p^{n+2}\,,$$ where we
  identify $\sigma_i$ with its class in the field $\fp_p$. We also
  denote $\varphi=\diag(\sigma)$ and we say that $\varphi$ is a
  diagonal automorphism.
\end{definition}

\begin{remark} \label{identif} Let $\sigma$ be the signature of an
  automorphism $\varphi$. The signature of $\varphi^a$ is
  $a\cdot\sigma$. Changing the representative of $\varphi$ in $\GL(V)$
  by $\xi^a\varphi$, corresponds to change the signature by
  $\sigma+a\cdot\mathbbm{1}$, where
  $\mathbbm{1}=(1,\ldots,1)$. Finally, the natural action of the
  symmetric group $S_{n+2}$ by permutation of the basis of
  $V$ corresponds to permutation of the signature $\sigma$.
\end{remark}

The following simple lemma is a key ingredient in our classification
of automorphisms of prime order of smooth cubic $n$-folds in
Section~\ref{class}.

\begin{lemma} \label{base} Let $X$ be a cubic hypersurface of
  ${\mathbb P}^{n+1}$, given by the homogeneous form $F\in S^3(V)$. If
  the degree of $F$ is smaller than $2$ in some of the variables
  $x_i$, then $X$ is singular.
\end{lemma}
\begin{proof}
  Without loss of generality, we may assume that the degree of $F$ is
  smaller than $2$ in the variable $x_0$. A direct computation shows
  that the point $(1:0:\ldots:0)$ is a singular point of $X$.
\end{proof}

\begin{remark} \label{fermat-def} %
  It is easy to see that for any $n\geq 2$, there exist at least one
  smooth cubic $n$-fold admitting an automorphism of order $2$ and
  $3$. For instance the Fermat $n$-fold $X=V(F)$, where
  $$F=x_0^3+x_1^3+\ldots+x_{n}^3+x_{n+1}^3\,,$$
  is smooth and admits the action of the symmetric group $S_{n+2}$ by
  permutation of the coordinates. Any transposition gives an
  automorphism of order 2 and any cycle of length 3 gives an
  automorphism of order 3.
\end{remark}

\begin{definition} \label{admis} We say that a prime number $p$ is
  \emph{admissible in dimension $n$} if either $p=2$ or there exists
    $\ell\in\{1,\ldots,n+2\}$ such that
$$(-2)^\ell\equiv 1  \mod p\,.$$
\end{definition}

This definition is justified by the following Theorem which shows that
the admissible primes in dimension $n$ are exactly those that are the
order of an automorphism of a smooth cubic $n$-fold.

\begin{theorem} \label{resultado} Let $n\geq 2$. A prime number $p$ is
  the order of an automorphism of a smooth cubic $n$-fold if and only
  if $p$ is admissible in dimension $n$.
\end{theorem}
\begin{proof}
  The theorem holds for the admissible primes $2$ and $3$ since for
  all $n\geq 2$, the
  Fermat $n$-fold admits automorphisms of order $2$ and $3$. In the
  following we assume that $p>3$.

  To prove the ``only if'' part, suppose that $F\in S^3(V)$ is a cubic
  form such that the $n$-fold $X=V(F)\subseteq \PP^{n+1}$ is smooth
  and admits an automorphism $\varphi$ of order $p$. Assume that
  $\varphi$ is diagonal and let
  $\sigma=(\sigma_0,\ldots,\sigma_{n+1})\in\fp_p^{n+2}$ its signature.

  We have $\varphi(F)=\xi^{a} F$. Let $b$ be such that $3b \equiv -a
  \mod p$. By changing the representation of $\varphi$ in $\GL(V)$ by
  $\xi^{b}\varphi$ we may assume that $\varphi(F)=F$.

  Let now $k_0$ be such that $\sigma_{k_0}\not\equiv 0$. By Lemma
  \ref{base}, $F$ contains a monomial $x_{k_0}^2x_{k_1}$ for some
  $k_1\in\{0,\ldots,n+1\}$. Since $p\neq 3$ we have $k_1\neq k_0$ and
  since $p\neq 2$ we have $\sigma_{k_1}\not\equiv 0 \mod p
  $. Furthermore, $F$ is invariant by the diagonal automorphism
  $\varphi$ so $2\sigma_{k_0}+\sigma_{k_1}\equiv 0 \mod p $, and so
  $$\sigma_{k_1}\equiv-2\sigma_{k_0}\, \mod p\,.$$

  With the same argument as above, for all $i\in\{2,\ldots,n+2\}$ we
  let $k_i\in\{0,\ldots,n+1\}$ be such that $x_{i-1}^2x_i$ is a
  monomial in $F$, so that
  $$\sigma_{i}\equiv -2\sigma_{i-1}\equiv (-2)^i\sigma_{k_0} \mod p
  ,\ \forall i\in\{2,\ldots,n+2\}\,,$$ and all of the $\sigma_{k_i}$
  are non-zero.

  Since $k_i\in\{0,\ldots,n+1\}$ there are at least two
  $i,j\in\{0,\ldots,n+2\}$, $i>j$ such that $k_i=k_j$. Thus
  $\sigma_{k_i}=\sigma_{k_j}$, and since $\sigma_i\equiv
  (-2)^i\sigma_{k_0}\mod p$ and $\sigma_j\equiv (-2)^j\sigma_{k_0}
  \mod p $, we have
  $$(-2)^{i-j}\equiv 1 \mod p\,,$$
  and the prime number $p$ is admissible in dimension $n$.

  To prove the converse statement, let $p>3$ be an admissible prime
  for dimension $n$. We let $\ell\in\{1,\ldots,n+2\}$ be such that
  $(-2)^\ell\equiv 1\mod p$ and consider the cubic form
  $$F=\sum_{i=1}^{\ell-1} x_{i-1}^2x_i +x_{\ell-1}^2x_0+
  \sum_{i=\ell}^{n+1}x_i^3\,.$$

  By construction, the cubic $F$ form admits the automorphism
  $\varphi=\diag(\sigma)$, where
  $$\sigma=\big(1,-2,(-2)^2,\ldots,(-2)^{\ell-1},
  \overbrace{0,\ldots,0}^{n+2-\ell}\big)\,.$$ A routine computation
  shows that $X=V(F)$ is smooth, proving the theorem.
\end{proof}

\begin{remark} \label{signa} %
  Let $\varphi=\diag(\sigma)$ be an automorphism of order $p>3$ prime
  of the smooth cubic $n$-fold $X=V(F)$. As in the proof of Theorem
  \ref{resultado}, assume that $\varphi(F)=F$ and let $\ell$ be as in
  Definition \ref{admis}. If $\sigma_0\neq 0$ is a component of the
  signature $\sigma$, then $(-2)^i\sigma_0$ is also a component of
  $\sigma$, $\forall i<\ell$. Furthermore, replacing $\varphi$ by
  $\varphi^{a}$, where $a\in\fp_p$ is such that $a\cdot\sigma_0\equiv
  1$, we can assume that $\sigma_0=1$.
\end{remark}

Theorem \ref{resultado} allows us to give, in the following corollary,
a bound for the prime numbers that appear as the order of an
automorphism of a smooth cubic $n$-fold.

\begin{corollary} \label{bound} If a prime number $p$ is the order of
  an automorphism of a smooth cubic $n$-fold, then $p<2^{n+1}$.
\end{corollary}
\begin{proof}
  Suppose that $p>2^{n+1}$. By Theorem~\ref{resultado}, $p$ is
  admissible in dimension $n$, and so
$$(-2)^{n+1}\equiv 1 \mod p, \quad \mbox{or}\quad (-2)^{n+2}\equiv 1 \mod p\,.$$
This yields
$$p=(-2)^{n+1}-1, \quad \mbox{or}\quad p=(-2)^{n+2}-1, \quad \mbox{or}\quad 2p=(-2)^{n+2}-1\,.$$
Since $-2\equiv 1\mod 3$, we have that $p$ is divisible by 3, which
provides a contradiction. 
\end{proof}

In \cite{szabo} a general bound is given for the order of a linear
automorphism of an $n$-dimensional projective variety of degree
$d$. Restricted to the case of cubic $n$-folds this bound is $p \leq
3^{n+1}$. Thus in this particular case our bound above is sharper.

\begin{definition} \label{klein} For any $n\geq 2$, we define the
  Klein $n$-fold as $X=V(F)\in\PP^{n+1}$, where
  $$F=x_0^2x_1+x_1^2x_2+\ldots+x_n^2x_{n+1}+x_{n+1}^2x_0\,.$$
\end{definition}

The group of automorphisms of the Klein threefold $X$ for was first
studied by Klein who showed that $\PSL(2,\fp_{11})<\aut(X)$
\cite{klein}.  Later, Adler showed that $\aut(X)=\PSL(2,\fp_{11})$
\cite{adler}. In the following theorem we show that if an $n$-fold $X$
admits an automorphism of prime order greater $2^n$, then $X$ is
isomorphic to the Klein $n$-fold.

\begin{theorem} \label{kuni} Let $X=V(F)$ be a smooth $n$-fold, $n\geq
  2$, admitting an automorphism $\varphi$ of order $p$ prime. If
  $p>2^n$ then $X$ is isomorphic to the Klein $n$-fold.
\end{theorem}

\begin{proof}
  Since $p>2^n$, by Corollary \ref{bound} $p$ is not admissible in
  dimension $n-1$ and so
  \begin{align}\label{lequ}
    (-2)^{n+2}\equiv 1\mod p\,.
  \end{align}

  By Remark \ref{signa}, we can assume that $\varphi(F)=F$ and
  $\varphi=\diag(\sigma)$, where
$$\sigma=(\sigma_0,\ldots,\sigma_{n+1})=\big(1,-2,4,\ldots,(-2)^{n+1}\big)\,.$$

Let $\EE\subset S^3(V)$ be the eigenspace associated to the eigenvalue
1 of the linear automorphism $\varphi:S^3(V)\rightarrow S^3(V)$, so
that $F\in\EE$. In the following we compute a basis for $\EE$.  For a
monomial $x_ix_jx_k$, $0\leq i\leq j \leq k\leq n+1$, in the basis of
$S^3(V)$ we have
$$x_ix_jx_k\in\EE\Leftrightarrow \sigma_i+\sigma_j+\sigma_k\equiv 0\Leftrightarrow (-2)^i+(-2)^j+(-2)^k\equiv 0 \mod p\,.$$

Clearly the only monomials with $i=j$ or $j=k$ contained in $\EE$ are
$x_{n+1}^2x_0$ and $x_{i}^2x_{i+1}$, $\forall i\in\{0,\ldots,n\}$. In
the following we assume that $i<j<k$. Multiplying by $(-2)^{-i}$ we
obtain
$$x_ix_jx_k\in\EE\Leftrightarrow 1+(-2)^{j-i}+(-2)^{k-i}\equiv 0 \mod p\,.$$

Let $j'=j-i$ and $k'=k-i$. If $k'<n$ then
$0<|1+(-2)^{j'}+(-2)^{k'}|<2^{n}<p$ and $x_ix_jx_k\notin\EE$. If
$k'=n+1$ then $j'\leq n$ and by \eqref{lequ}
$$1+(-2)^{j'}+(-2)^{k'}\equiv 0\mod p\Leftrightarrow -2+(-2)^{j'+1}+1\equiv 0\mod p\,.$$
This is only possible if $(-2)^{j'+1}-1=p$ or $(-2)^{j'+1}-1=2p$, but
as in the proof of Corollary \ref{bound}, $(-2)^{k}-1$ is divisible by
$3$ for any $k\in\ZZ$, so $x_ix_jx_k\notin\EE$.

If $k'=n$ then $j'\leq n-1$ and by \eqref{lequ}
$$1+(-2)^{j'}+(-2)^{k'}\equiv 0\mod p\Leftrightarrow 4+(-2)^{j'+2}+1\equiv 0\mod p\,.$$
Again this is only possible if $(-2)^{j'+2}+5=p$ or
$(-2)^{j'+2}+5=2p$. The same argument as before gives $x_ix_jx_k\notin
\EE$.

We have shown that $\EE=\big\langle x_{n+1}^2x_0, x_{i}^2x_{i+1},\
\forall i\in\{0,\ldots,n\}\big\rangle$, and so
$$F=\sum_{i=0}^{n} a_ix_{i}^2x_{i+1} +a_{n+1}x_{n+1}^2x_0,\quad a_i\in\CC\,.$$
Since $X$ is smooth, by Lemma \ref{base} all of the $a_i$ above are
non-zero and applying a linear change of coordinates we can put
$$F=\sum_{i=0}^{n} x_{i}^2x_{i+1} +x_{n+1}^2x_0\,.$$
\end{proof}

The particular case of $n=3$ in Theorem \ref{kuni}, was shown by
Roulleau \cite{roulleau}. In that article it is shown that the Klein
threefold is the only cubic threefold admitting an automorphism of
order 11.

The criterion in Theorem \ref{resultado} is easily computable. In the
following table we give the list of admissible prime numbers for
$n\leq 10$.

\vspace{1ex}
\begin{center}
  \begin{tabular}{ | c | c | }
    \hline	
    $n$ & admissible primes \\
    \hline		
    2 & $2, 3, 5$ \\
    3 & $2, 3, 5, 11$  \\
    4 & $2, 3, 5, 7, 11$ \\
    5 & $2, 3, 5, 7, 11, 43$ \\
    6 & $2, 3, 5, 7, 11, 17, 43$ \\
    7 & $2, 3, 5, 7, 11, 17, 19, 43$ \\
    8 & $2, 3, 5, 7, 11, 17, 19, 31, 43$ \\
    9 & $2, 3, 5, 7, 11, 17, 19, 31, 43, 683$ \\
    10 & $2, 3, 5, 7, 11, 13, 17, 19, 31, 43, 683$ \\
    \hline
  \end{tabular}
\end{center}
\vspace{1ex}

In the following table, we give the maximal admissible prime number
$p$ for $11\leq n\leq 20$.

\vspace{1ex}
\begin{center}
  \begin{tabular}{ | c | c | c | c | c | c | c | c | c | c | c | }
    \hline	
    $n$ & 11 & 12 & 13 & 14 & 15 & 16 & 17 & 18 & 19 & 20 \\
    \hline
    $p$ & 2731 & 2731 & 2731 & 2731 & 43691 & 43691 & 174763 & 174763 & 174763 & 174763 \\
    \hline
  \end{tabular}
\end{center}
\vspace{1ex}

\begin{remark} \label{rk:klein-new}
  By the proof of Theorem \ref{resultado}, whenever $p$ is admissible
  for dimension $n$ and not for $n-1$, the Klein $n$-fold $X$ admits
  an automorphism of order $p$. Furthermore, If $p>2^n$, then by
  Theorem \ref{kuni} $X$ is the only smooth cubic $n$-fold admitting
  such an automorphism. This is the case for the maximal admissible
  prime for $n=2,3,5,9,11,15,17$.
\end{remark}

\section{Classification of automorphisms of prime order}
\label{class}

In this section we apply the results in Section \ref{main} to provide
a computational method to classify all the smooth cubic $n$-folds
admitting an automorphism or order $p$ prime, or what is the same,
$n$-folds that admits the action of a cyclic group of order $p$.

The classification of automorphism of smooth cubic surfaces was done
by Segre \cite{segre} and later revised by Dolgachev and Verra
\cite{dolga} with more sofisticated methods, in this case all the
computations can be carried out by hand. As an example, we apply our
method to the cases of threefolds and fourfolds since in greater
dimension the lists are rather long.

% Recall that the space of cubic forms in $\PP^4$ and $\PP^5$ are the
% projective spaces $\PP^{34}$ and $\PP^{55}$, respectively.  The smooth
% cubic forms of $\PP^4$ and $\PP^5$ are the open subsets ${\mathcal
%   U}_3 \subset \PP^{34}$ and ${\mathcal U}_4 \subset \PP^{55}$,
% respectively. Given two cubic forms $F_1$, $F_2$ the corresponding
% cubic hypersurfaces $V(F_1)$ and $V(F_2)$ are isomorphic if and only
% if they are equivalent under the action of $\PGL(5,\CC)$ and of
% $\PGL(6,\CC)$, respectively. The isomorphisms classes of cubics
% threefolds and fourfolds correspond to elements of the orbit spaces
% $${\mathfrak H}_3= {\mathcal U}_3/\PGL(5,\CC), \qquad {\mathfrak H}_4=
% {\mathcal U}_4/\PGL(6,\CC)\,.$$ %
% In Geometric Invariant Theory it is proved that ${\mathfrak H}_3$
% and ${\mathfrak H}_4$ have the structure of quasi-projective varieties
% of dimension $10$ and $20$.

All the possible signatures for an automorphism of order $p$ prime of
a cubic $n$-fold, $n\geq 2$ are given by $\fp_p^{n+2}$, but many of
them represent identical or conjugated cyclic groups $\ZZ_p$ on
$\PGL(V)$.

\begin{enumerate}[$(i)$]
\item Two diagonal automorphism
  $\diag(\sigma),\diag(\sigma')\in\GL(V)$ are conjugated if and only
  if there exists a permutation $\pi\in S_{n+2}$ such that
  $\sigma'=\pi(\sigma)$.

\item By Remark \ref{identif}, two diagonal automorphism
  $\diag(\sigma),\diag(\sigma')\in\PGL(V)$ span the same cyclic group
  if and only if $\sigma'=a\cdot\sigma+b\cdot\mathbbm{1}$, for some
  $a\in\fp_p^*$ and some $b\in\fp_p$.
\end{enumerate}

\begin{definition}
  We define an equivalence relation $\sim$ in $\fp_p^{n+2}$ by
  $\sigma\sim\sigma'$ if and only if
  $\sigma'=a\cdot\pi(\sigma)+b\cdot\mathbbm{1}$, for some $\pi\in
  S_{n+2}$, some $a\in\fp_p^*$, and some $b\in\fp_p$.
\end{definition}

\begin{remark} \
  \begin{enumerate}[$(i)$]
  \item Each class in $(\fp_p^{n+2}/\sim)$ represent a conjugacy class
    of cyclic groups of order $p$ in $\PGL(V)$. The class of zero
    $\overline{0}$ represents the identity of $\PP^{n+1}$.
  \item The map $\sigma\mapsto a\cdot\pi(\sigma)+b\cdot\mathbbm{1}$,
    for some $\pi\in S_{n+2}$, some $a\in\fp_p^*$, and some
    $b\in\fp_p$ is an automorphism of $\fp_p^{n+2}$ regarded as the
    affine space over the field $\fp_p$. Let $G<\aut(\fp_p^{n+2})$ be
    the finite group spanned by all such automorphisms. The
    equivalence classes of the relation $\sim$ are given by the orbits
    of the action of $G$ in $\fp_p^{n+2}$. Hence,
    $$(\fp_p^{n+2}/\sim)=\fp_p^{n+2}//G\,.$$
  \end{enumerate}
\end{remark}

For any signature $\sigma\in\fp_p^{n+2}$ we let $\EE_{\sigma}<S^3(V)$
be the eigenspace associated to the eigenvalue $1$ of the automorphism
$\varphi:S^3(V)\rightarrow S^3(V)$.

\begin{remark} \label{pn3} Let $X=V(F)$ be a smooth cubic $n$-fold
  admitting an automorphism $\varphi=\diag(\sigma)$ of order $p$
  prime. Whenever $p\neq 3$, as in the proof of Theorem
  \ref{resultado}, we can assume that $\varphi(F)=F$ i.e.,
  $F\in\EE_\sigma$.
\end{remark}

\begin{method}
  For any admissible prime $p\neq 3$ in dimension $n$,
  \begin{enumerate}[$(ii)$]
  \item Compute the set of signatures $R'\subset\fp_p^{n+2}$ such that
    there exists a cubic form $F\in \EE_{\sigma}$ with $X=V(F)$
    smooth.
  \item Let $\widetilde{R}=R'/\sim$ and compute a set $R$ of
    representatives of $\widetilde{R}\setminus\{\overline{0}\}$.
  \end{enumerate}
\end{method}

By construction, our method provides a set of signatures $R$ such that
for every smooth cubic $n$-fold $X=V(F)$ admitting an automorphism
$\varphi$ order $p$, there exists one and only one $\sigma\in R$ such
that after a linear change of coordinates $F\in \EE_{\sigma}$ and
$\diag(\sigma)$ is a generator of $\langle\varphi\rangle$.

Remark \ref{pn3} does not hold for $p=3$. In order to apply the same
method, in the first step we let $R'$ be the set of signatures
$\sigma\in\fp_3^{n+2}$ such that there exists a cubic form $F\in
\EE_{\sigma}$ with $X=V(F)$ smooth or there exists a cubic form $F$ in
the eigenspace associated to eigenvalue $\xi$ with $X=V(F)$ smooth,
where $\xi$ is a principal cubic root of the unity.

All of this procedure can be implemented on a software such as
Maple. Lemma \ref{base} is used to discard most of the signatures
$\sigma$ whose $\EE_{\sigma}$ does not contain a smooth cubic
$n$-fold, the remaining signatures that does not contain a smooth
$n$-fold can be easily eliminated by hand.

Given a signature $\sigma\in R$, we let $\varphi=\diag(\sigma)$. For a
generic cubic forms $F\in \EE_{\sigma}$ the cubic $n$-fold $X=V(F)$ is
smooth. The dimension of the family of cubic $n$-folds given by
$\EE_{\sigma}$ in the moduli space $\mathfrak{H}_n$ of smooth cubic
$n$-folds is
\begin{align}
  D=\dim \EE_{\sigma} -\dim
  \operatorname{N}_{\GL(V)}(\langle\varphi\rangle)\,.
\end{align}
To compute the dimension of the normalizer
$\operatorname{N}_{\GL(V)}(\langle\varphi\rangle)$ we have the
following simple lemma.

\begin{lemma}
  Let $\sigma\in\fp_p^{n+2}$ and let $\varphi=\diag(\sigma)$. If $n_j$
  is the number of times $j\in\fp_p$ appears in $\sigma$, then
$$\dim\operatorname{N}_{\GL(V)}(\langle\varphi\rangle)=n_0^2+...+n_{p-1}^2\,.$$
\end{lemma}

Our method applied to cubic surfaces gives for every prime number $p$
the corresponding result contained in the lists of Segre
\cite{segre}, and Dolgachev and Verra \cite{dolga}. Theorems
\ref{tthre} and \ref{tfour} contains the results of the method
described in this section applied to threefolds and fourfolds,
respectively.

\begin{theorem} \label{tthre} %
  Let $X=V(F)$ be a smooth cubic threefold in $\PP^4$ that admits an
  automorphism $\varphi$ of order $p$ prime, then after a linear
  change of coordinates that diagonalizes $\varphi$, $F$ is given in
  the following list, a generator of $\langle \varphi \rangle$ is
  given by
  $$\diag(\sigma):\PP^4\rightarrow \PP^4,\qquad
  (x_0:\ldots:x_{4})\mapsto(\xi^{\sigma_0}x_0:\ldots:\xi^{\sigma_{4}}x_{4})\,,$$
  and $D$ is the dimension of the family of smooth cubic threefold
  that admits this automorphism.
  \begin{description}
  \item[$\TT_2^1$] $p=2,\ \sigma=(0,0,0,0,1),\ D=7,$
$$F=x_4^2L_1(x_0,x_1,x_2,x_3)+L_3(x_0,x_1,x_2,x_3)$$

\item[$\TT_2^2$] $p=2,\ \sigma=(0,0,0,1,1),\ D=6,$
$$F=x_0L_2(x_3,x_4)+x_1M_2(x_3,x_4)+x_2N_2(x_3,x_4)+L_3(x_0,x_1,x_2)$$

\item[$\TT_3^1$] $p=3,\ \sigma=(0,0,0,0,1),\ D=4,$
$$F=L_3(x_0,x_1,x_2,x_3)+x_4^3$$

\item[$\TT_3^2$] $p=3,\ \sigma=(0,0,0,1,1),\ D=1,$
$$F=L_3(x_0,x_1,x_2)+M_3(x_3,x_4)$$

\item[$\TT_3^3$] $p=3,\ \sigma=(0,0,0,1,2),\ D=4,$
$$F=L_3(x_0,x_1,x_2)+x_3x_4L_1(x_0,x_1,x_2)+x_3^3+x_4^3$$

\item[$\TT_3^4$] $p=3,\ \sigma=(0,0,1,1,2),\ D=2,$
$$F=L_3(x_0,x_1)+M_3(x_2,x_3)+x_4^3+L_1(x_0,x_1)M_1(x_2,x_3)x_4$$

\item[$\TT_5^1$] $p=5,\ \sigma=(0,1,2,3,4),\ D=2,$
$$F=a_1x_0^3+a_2x_0x_1x_4+a_3x_0x_2x_3+a_4x_1^2x_3+a_5x_1x_2^2+a_6x_2x_4^2+a_7x_3^2x_4$$

\item[$\TT_{11}^1$] $p=11,\ \sigma=(1,3,4,5,9),\ D=0,$
$$F=x_0^2x_4+x_3^2x_0+x_1^2x_3+x_1x_2^2+x_2x_4^2$$
\end{description}
Here the $L_i,M_i$, and $N_i$ are forms of degree $i$.
\end{theorem}

\begin{proposition}
  The Fermat cubic threefold belongs to all the above families except
  ${\TT}_{11}^1$. The Klein cubic threefold belongs to the families
  ${\TT}_2^2$,${\TT}_3^4$, ${\TT}_5^1$ and ${\TT}_{11}^1$.
\end{proposition}

\begin{proof}
  The automorphisms group $\aut(X)$ of the Fermat threefold $X$ is
  isomorphic to an $S_5$-extension of ${\mathbb Z}_3^4$, where $S_5$
  acts in $\PP^4$ by permutation of the coordinates and ${\mathbb Z}_3^4$
  acts in $\PP^4$ by multiplication in each coordinate by a cubic root
  of the unit \cite{kontogeorgis}. 

  It is easy to produce matrices in $\aut(X)$ with the same spectrum
  of the automorphism of all the above families except
  ${\TT}_{11}^1$. Furthermore, since this group does not contain any
  element of order $11$, it follows that $X$ does not belongs to the
  family ${\TT}_{11}^1$.

  The automorphisms group $\aut(X)$ of the Klein threefold $X$ is
  isomorphic to $\PSL(2,\fp_{11})$ \cite{adler}.
  % generated by the following matrices in $\PGL(5,\CC)$ :
  % $$R=
  % \begin{pmatrix}
  %   \xi & 0 & 0 & 0 & 0 \\
  %   0 & \xi^9 & 0 & 0 & 0 \\
  %   0 & 0 & \xi^4 & 0 & 0 \\
  %   0 & 0 & 0 & \xi^3 & 0 \\
  %   0 & 0 & 0 & 0 & \xi^5
  % \end{pmatrix},\quad
  % S=
  % \begin{pmatrix}
  %   0 & 1 & 0 & 0 & 0 \\
  %   0 & 0 & 1 & 0 & 0 \\
  %   0 & 0 & 0 & 1 & 0 \\
  %   0 & 0 & 0 & 0 & 1 \\
  %   1 & 0 & 0 & 0 & 0
  % \end{pmatrix}\,, \quad
  % $$
  % $$
  % \xi^{11}=1,\quad\mbox{and}\quad T=\sum_{k=1}^5
  % S^k(R^2-R^{-2})S^k\,.$$
  %
  The conjugacy classes of element of prime order in
  $\PSL(2,\fp_{11})$ are as follows: two conjugacy classes of elements
  of order $11$; two conjugacy classes of elements of order $5$; one
  conjugacy class of elements of order $3$; and one conjugacy class of
  elements of order $2$.

  It follows immediately that $X$ belongs to the families ${\TT}_5^1$
  and ${\TT}_{11}^1$. Since there is only one conjugacy class of
  orders 2 and 3, respectively, then $X$ belongs to one and only one
  of the families with automorphisms of order 2 and 3,
  respectively. An easy computation shows that the Klein
  threefold belongs to ${\TT}_2^2$ and ${\TT}_3^4$.
\end{proof}

\begin{remark}
  The above proposition shows, in particular, that a generic element
  in all the families in Theorem~\ref{tthre} is smooth.
\end{remark}

\begin{theorem} \label{tfour} Let $X=V(F)$ be a smooth cubic fourfold
  in $\PP^5$ that admits an automorphism $\varphi$ of order $p$ prime,
  then after a linear change of coordinates that diagonalizes
  $\varphi$, $F$ is given in the following list, a generator of
  $\langle \varphi \rangle$ is given by
$$\diag(\sigma):\PP^5\rightarrow \PP^5,\qquad (x_0:\ldots:x_{5})\mapsto(\xi^{\sigma_0}x_0:\ldots:\xi^{\sigma_{5}}x_{5})\,,$$
and $D$ is the dimension of the family of smooth cubic fourfolds that
admits this automorphism.

\begin{description}
\item[$\FF_2^1$] $p=2,\ \sigma=(0,0,0,0,0,1),\ D=14,$
$$F=L_3(x_0,x_1,x_2,x_3,x_4)+x_5^2L_1(x_0,\cdots,x_4)\,.$$

\item[$\FF_2^2$] $p=2,\ \sigma=(0,0,0,0,1,1),\ D=12,$
$$F=L_3(x_0,x_1,x_2,x_3)+x_4^2L_1(x_0,x_1,x_2,x_3)+x_4x_5M_1(x_0,x_1,x_2,x_3)+x_5^2N_1(x_0,x_1,x_2,x_3)\,.$$

\item[$\FF_2^3$] $p=2,\ \sigma=(0,0,0,1,1,1),\ D=10,$
$$F=L_3(x_0,x_1,x_2)+x_0L_2(x_3,x_4,x_5)+x_1M_2(x_3,x_4,x_5)+x_2N_2(x_3,x_4,x_5)\,.$$

\item[$\FF_3^1$] $p=3,\ \sigma=(0,0,0,0,0,1),\ D=10,$
$$F=L_3(x_0,x_1,x_2,x_3,x_4)+x_5^3\,.$$

\item[$\FF_3^2$] $p=3,\ \sigma=(0,0,0,0,1,1),\ D=4,$
$$F=L_3(x_0,x_1,x_2,x_3)+M_3(x_4,x_5)\,.$$

\item[$\FF_3^3$] $p=3,\ \sigma=(0,0,0,0,1,2),\ D=8,$
$$F=L_3(x_0,x_1,x_2,x_3)+x_4^3+x_5^3+x_4x_5L_1(x_0,x_1,x_2,x_3)\,.$$

\item[$\FF_3^4$] $p=3,\ \sigma=(0,0,0,1,1,1),\ D=2,$
$$F=L_3(x_0,x_1,x_2)+M_3(x_3,x_4,x_5)\,.$$

\item[$\FF_3^5$] $p=3,\ \sigma=(0,0,0,1,1,2),\ D=7,$
$$F=L_3(x_0,x_1,x_2)+M_3(x_3,x_4)+x_5^3+x_3x_5L_1(x_0,x_1,x_2)+x_4x_5M_1(x_0,x_1,x_2)\,.$$

\item[$\FF_3^6$] $p=3,\ \sigma=(0,0,1,1,2,2),\ D=8,$
$$F=L_3(x_0,x_1)+M_3(x_2,x_3)+N_3(x_4,x_5)+\sum_{i=1,2;\,j=3,4;\,k=5,6} a_{ijk}x_ix_jx_k\,.$$

\item[$\FF_3^7$] $p=3,\ \sigma=(0,0,1,1,2,2),\ D=6,$
$$F=x_2L_2(x_0,x_1)+x_3M_2(x_0,x_1)+x_4^2L_1(x_0,x_1)+$$$$x_4x_5M_1(x_0,x_1)+x_5^2N_1(x_0,x_1)+x_4N_2(x_2,x_3)+x_5O_2(x_2,x_3)\,.$$

\item[$\FF_5^1$] $p=5,\ \sigma=(0,0,1,2,3,4),\ D=4,$
$$F=L_3(x_0,x_1)+x_2x_5L_1(x_0,x_1)+x_3x_4M_1(x_0,x_1)+x_2^2x_4+x_2x_3^2+x_3x_5^2+x_4^2x_5\,.$$

\item[$\FF_5^2$] $p=5,\ \sigma=(1,1,2,2,3,4),\ D=2,$
$$F=x_0L_2(x_2,x_3)+x_1M_2(x_2,x_3)+x_4N_2(x_0,x_1)+x_5^2M_1(x_2,x_3)+x_4^2x_5\,.$$

\item[$\FF_7^1$] $p=7,\ \sigma=(1,2,3,4,5,6),\ D=2,$
$$F=x_0^2x_4+x_1^2x_2+x_0x_2^2+x_3^2x_5+x_3x_4^2+x_1x_5^2+ax_0x_1x_3+bx_2x_4x_5\,.$$

\item[$\FF_{11}^1$] $p=11,\ \sigma=(0,1,3,4,5,9),\ D=0,$
$$F=x_0^3+x_1^2x_5+x_2^2x_4+x_2x_3^2+x_1x_4^2+x_3x_5^2\,.$$
\end{description}
Here $L_i$, $M_i$, $N_i$ and $O_i$ are forms of degree $i$.
\end{theorem}

\begin{proposition}
  The Fermat cubic fourfold belongs to all the above families except
  ${\FF}_{3}^7$, ${\FF}_{5}^2$, ${\FF}_{7}^1$, and ${\FF}_{11}^1$. 
\end{proposition}

\begin{proof}
  The automorphisms group $\aut(X)$ of the Fermat fourfold $X$ given
  by $F=x_0^3+x_1^3+x_2^3+x_3^3+x_4^3+x_5^3$ is isomorphic to an
  $S_6$-extension of ${\mathbb Z}_3^5$, where $S_6$ acts in $\PP^5$ by
  permutation of the coordinates and ${\mathbb Z}_3^5$ acts in $\PP^5$
  by multiplication in each coordinate by a cubic root of the unit
  \cite{kontogeorgis}.

  It is easy to produce matrices in $\aut(X)$ with the same spectrum
  of the automorphism of all the above families except ${\FF}_{3}^7$,
  ${\FF}_{5}^2$, ${\FF}_{7}^1$, and ${\FF}_{11}^1$.  Furthermore, $X$
  does not admit automorphisms of order 7 or 11, thus $X$ does not
  belong to ${\FF}_{7}^1$ or ${\FF}_{11}^1$.

  For the remaining two cases we have: $X$ does not belong to the
  family ${\FF}_{5}^2$ since there is only one conjugacy class of
  elements of order 5 in $S_6$ and it spectrum corresponds to the
  family ${\FF}_{5}^1$; and $X$ does not belong to ${\FF}_{3}^7$ since
  every automorphism of order 3 of $X$ leaves $F$ invariant and this
  is not the case if the family ${\FF}_{3}^7$.
\end{proof}

\begin{remark}
  To obtain a smooth fourfold in the remaining 4 families we proceed
  as follows: by Remark~\ref{rk:klein-new} the Klein fourfold belongs
  to ${\FF}_{7}^1$; the family ${\FF}_{11}^1$ is reduced one fourfold,
  a triple covering of the Klein threefold, which is smooth; and for
  the remaining two families we can check with a software such as
  Maple that the fourfold obtained by putting all the parameters equal
  to one is smooth.
\end{remark}

\section{An application to abelian varieties}
\label{abvar}

Let $X$ be a smooth compact Kahler manifold of dimension $n$ and let
$q\leq n$ be a positive integer. We denote by $\JJ_q(X)$ its
$q$-intermediate Griffiths jacobian. $\JJ_q(X)$ is a complex torus. If
$n=2q-1$ is odd $\JJ_q(X):=\JJ(X)$ is called the intermediate jacobian
of $X$. If $X$ is a smooth hypersurface of degree $d$ of ${\PP}^{n+1}$
it s known \cite{deligne} that $\JJ(X)$ is a non-trivial principally
polarized abelian variety (p.p.a.v.) if and only if $n=1$ and $d\geq
3$, $n=3$ and $d=3,4$, or $n=5$ and $d=3$.

Let $\AAA_g$ be the moduli space of p.p.a.v.'s of dimension $g$. It
follows that $\JJ(X_3) \in \AAA_5$ and $\JJ(X_5) \in \AAA_{21}$, where
$X_3$ and $X_5$ are the Klein cubic threefold and $5$-fold,
respectively (see Definition~\ref{klein}).

The Klein $3$-fold $X_3$ admits an automorphism of prime order $p=11$,
see Theorem~\ref{tthre}. By Theorems~\ref{resultado} and \ref{kuni},
the Klein $5$-fold $X_5$ is the only $5$-fold admitting an
automorphism of prime order $p=43$ given by
$$\varphi(x_0,\ldots,x_6)= (\xi x_0,\xi^{41}x_1,\xi^{4}x_2,
\xi^{35}x_3,\xi^{16}x_4,\xi^{11}x_5,\xi^{21}x_6)\,.$$

It follows from \cite{victor} that $\JJ(X_3)$ and $\JJ(X_5)$ are
p.p.a.v. of complex multiplication type and that they are
zero-dimensional components of the singular locus of $\AAA_5$ and
$\AAA_{21}$, respectively.

It was proved in \cite{victor} that $\JJ(X_3)$ is a non-isolated
singular point of $\AAA_5$.  The automorphism $\varphi$ of $X_5$
induces an automorphism $\widetilde{\varphi}$ of the tangent space
$T_0\JJ(X_5)$. A routine computation shows that the spectrum of
$\widetilde{\varphi}$ is
$$\left\{\xi^{2},\xi^{3},\xi^{5},\xi^{8},\xi^{9},\xi^{12},\xi^{13},
  \xi^{14},\xi^{15},\xi^{17},
  \xi^{19},\xi^{20},\xi^{22},\xi^{25},\xi^{27},\xi^{32},\xi^{33},
  \xi^{36},\xi^{37},\xi^{39},\xi^{42}\right\}\,.$$

The spectrum of $\widetilde{\varphi}$ is stabilized by the map
$$\psi:\fp_{43}^*\rightarrow\fp_{43}^*,\qquad a\mapsto \psi(a)=a^{11}\,.$$
It follows by \cite[Theorem 1]{victor} that $\JJ(X_5)$ is also a
non-isolated singular point of $\AAA_{21}$.

\end{document}